\documentclass[12pt]{amsart}
\usepackage{amssymb}
\usepackage{amsmath,latexsym}
\usepackage{amsfonts}
\usepackage{amsthm}
\usepackage{hyperref}
\usepackage{eucal}
\usepackage{mathrsfs}

\allowdisplaybreaks[4]

\newcommand{\vol}{{\rm vol}}
\newcommand{\Vol}{{\rm Vol}}

\newcommand{\so}{\mbox{${\mathfrak s \mathfrak o}$}}

\newcommand{\C}{\mbox{${\mathbb C}$}}

\newcommand{\I}{\mbox{${\mathbb I}$}}

\newcommand{\PP}{\mbox{${\mathbb P}$}}

\newcommand{\R}{\mbox{${\mathbb R}$}}

\newcommand{\tr}{{\rm tr}}

\newcommand{\Ric}{{\rm Ric}}

\newcommand{\SO}{{\rm SO}}
\newcommand{\SU}{{\rm SU}}
\newcommand{\Sp}{{\rm Sp}}

\newcommand{\U}{{\rm U}}

\numberwithin{equation}{section}

\newtheorem{thm}{Theorem}[section]

\newtheorem{lem}[thm]{Lemma}
\newtheorem{prop}[thm]{Proposition}
\newtheorem{cor}[thm]{Corollary}

\newtheorem{remark}[thm]{Remark}
\newtheorem{ex}[thm]{Example}
\newenvironment{rmk}{\begin{remark} \em}{\end{remark}}
\newenvironment{example}{\begin{ex} \em}{\end{ex}}

\begin{document}

\title[]{Stability of Einstein metrics on fiber bundles}

\author{Changliang Wang}
\address{Department of Mathematics and Statistics, McMaster University, Hamilton, Ontario, L8S 4K1 Canada}
\email{wangc114@math.mcmaster.ca}

\author{Y. K. Wang}
\address {Department of Mathematics and Statistics, McMaster University, Hamilton, Ontario, L8S 4K1 Canada}
\email{wang@mcmaster.ca}



\date{revised \today}

\begin{abstract}
We study the linear stability of Einstein metrics of Riemannian submersion type. First, we derive a general instability condition for such Einstein metrics and provide some applications. Then we study instability arising from Riemannian product structures on the base. As an application, we estimate the coindex of the Einstein metrics constructed in \cite{WZ90} and \cite{Wan92}. Finally, we investigate more closely the linear stability of Einstein metrics from circle bundle constructions and obtain a rigidity result for linearly stable Einstein metrics of this type.
\end{abstract}

\maketitle

\section{\bf Introduction}
Einstein metrics appear naturally in some geometric variational problems. For example, Einstein metrics on
a closed manifold $M^{n}$ are precisely the critical points of the normalized total scalar curvature functional
\begin{equation}
\widetilde{\bf S}(g)=\frac{1}{(\Vol(M, g))^{\frac{n-2}{n}}}\int_{M}s_{g}\,d\vol_{g},
\end{equation}
where $s_{g}$ is the scalar curvature of the Riemannian metric $g$. At such Einstein metrics, the second variation formula of the functional $\widetilde{\bf S}$ was first derived and studied by M. Berger \cite{Ber70} and N. Koiso \cite{Koi78}, \cite{Koi79}.
This formula shows that an Einstein metric is always a saddle point. Indeed, the second variation of $\widetilde{\bf S}$ is non-negative in the direction of conformal changes, by the theorem of Lichnerowicz-Obata. On the other hand,  along variations that are transverse to diffeomorphisms and conformal changes, i.e., for $h \in C^{\infty}(M, S^{2}(T^{*}M))$ with $\tr_{g} h=0$ and $\delta_{g}h=0$, the second variation of  $\widetilde{\bf S}$ at $g$ is given by
\begin{equation}
-\frac{1}{2(\Vol(M, g))^{\frac{n-2}{n}}}\int_{M}\langle\nabla^{*}\nabla h-2\mathring{R}h, h\rangle \,d\vol_{g},
\end{equation}
where $\delta_{g}h$ is the divergence of $h$, and $(\mathring{R}h)_{ij}:=R_{ikjl}h^{kl}$. Then by basic facts about the spectrum of elliptic operators on compact manifolds, there always exists an infinite-dimensional subspace of the transverse traceless
symmetric $2$-tensors (in short, the {\em TT-tensors}) on which the second variation of $\widetilde{\bf S}$ is negative definite. However,
there may also exist TT-tensors along which the second variation is positive. The dimension of a maximal subspace of
such TT-tensors is called the {\it coindex} of $g$ and this number is finite by elliptic theory.
This leads to the linear stability problem for Einstein metrics.

A closed Einstein manifold $(M^{n}, g)$ is called {\em linearly stable} if $\langle \nabla^{*}\nabla h - 2 \mathring{R} h, h\rangle_{L^{2}(M)}\geq 0$ for all TT-tensors $h$ and {\em linearly unstable} if otherwise. For Einstein metrics on noncompact manifolds
linear stability has also been considered by restricting one's attention to compactly supported divergence-free symmetric $2$-tensors
in the second variation formula. In this paper we shall refer to linear stability (resp. linear instability) simply as stability
(resp. instability).

Einstein metrics with non-positive scalar curvature tend more often to be stable. Examples include Einstein metrics with negative sectional curvature \cite{Koi78} and compact K\"{a}hler Einstein metrics with non-positive scalar curvature \cite{DWW07}. Complete Riemannian manifolds with imaginary Killing spinors, which must be non-compact and have negative Einstein constant, are also stable, see \cite{Kro17} and \cite{Wan17}.  In the Ricci-flat case, a closed Riemannian manifold with a non-trivial parallel spinor is stable \cite{DWW05}. By contrast, there are very few examples of stable Einstein manifolds with positive scalar curvature. All known examples are in fact irreducible symmetric spaces of compact type, but there is also an infinite subfamily of these which are unstable, see \cite{Koi80} and \cite{CH15}.

Recall that for TT-tensors at an Einstein metric with Einstein constant $E$,  the operator
$\nabla^{*}\nabla  - 2 \mathring{R}$ coincides with $-(\Delta_L + 2E \,\I)$ where $\Delta_L$ is the analyst's Lichnerowicz Laplacian.
Hence the definition of linear stability we use here allows for the vanishing of $ \nabla^{*}\nabla h - 2 \mathring{R} h$
(infinitesimal Einstein deformations) and does not imply that the functional $ \widetilde{\bf S}$
is locally maximized when restricted to TT-variations. On the other hand, linear instability along a nonzero TT-tensor gives
a non-trivial variation along which $ \widetilde{\bf S}$ has a local minimum. When $E > 0$, Theorem 1.3 in \cite{Kro15} then
implies that such an Einstein metric is {\em dynamically unstable}. Therefore, all linearly unstable Einstein metrics discussed in this paper are dynamically unstable, for example Einstein metrics investigated in Corollaries $\ref{TorusBundleInstability}$ and $\ref{QuaternionicInstability}$. For more information about dynamical instability
under the Ricci flow, see in addition \cite{Ye93}, \cite{Ses06}, \cite{Has12}, and \cite{HM14}.

Einstein metrics can also be characterized through variation problems for other interesting functionals, e.g., Perelman's $\lambda$- and $\nu$-functionals \cite{Per02}, and the $\nu_{+}$-functional introduced in \cite{FIN05}. Interestingly, the second variation formulas of these functionals lead to the same stability condition at Einstein metrics along TT-tensor directions, but lead to a different stability condition along conformal change directions, see \cite{Zhu11}, \cite{CZ12}, and \cite{CH15}. Since we work with TT-tensors in this paper, our results
(described below) can also be interpreted in terms of these other functionals.

In this paper, we shall investigate the instability of Einstein metrics with positive Einstein constant on the total spaces of Riemannian submersions with totally geodesic fibers. Einstein metrics of this type including their moduli and stability are interesting not only in geometry but also for supergravity theories in physics.  For basic facts about Riemannian submersions, see \cite{ONe66} and \cite{Bes87}. Our setup
is as follows. Let $\pi: (M^{n+r}, g)\rightarrow (B^{n}, \check{g})$ be a Riemannian submersion with totally geodesic fibers, i.e., O'Neill's $T$-tensor vanishes. The Einstein conditions for the metric $g$ are given in 9.61 of \cite{Bes87}. They imply that the scalar curvature $\hat{s}$ of the induced metric $\hat{g}$ on fibers and the norm $\|A\|$ of O'Neill's $A$-tensor are constant on $M$, and the scalar curvature $\check{s}$ of $\check{g}$ is constant on $B$.

When $g$ is Einstein and $\pi$ is a trivial Riemannian submersion, i.e., $(M^{n+r}, g)=(F^{r}, \hat{g})\times (B^{n}, \check{g})$, it is
well-known that $g$ is unstable with the canonical destabilizing direction $\frac{1}{r}\hat{g}-\frac{1}{n}\check{g}$. This motivates us to examine the stability condition along the direction $\frac{1}{r}g-\frac{n+r}{nr}\pi^{*}\check{g}$ in the general case. From the Einstein conditions for $g$, one sees that this is a TT-tensor on $M$ and we obtain the following instability condition.

\begin{thm}\label{ProductTypeInstability}
\begin{equation*}
\left\langle (\nabla^{*}\nabla-2\mathring{R})\left(g-\frac{n+r}{n}\pi^{*}\check{g}\right), \left(g-\frac{n+r}{n}\pi^{*}\check{g}\right)\right\rangle_{L^{2}(M)}=\Vol(M, g)\frac{2(n+r)}{n^{2}}(r\check{s}-2n\hat{s}).
\end{equation*}

In particular, if $r\check{s}< 2n\hat{s},$ the Einstein metric $g$ is unstable.
\end{thm}

The proof of the above theorem together with some applications are given in \S 2.

\medskip

 We like to mention that the instability of Einstein metrics on compact warped product manifolds has been studied in \cite{BHM17}.
 This is a ``dual" situation as these warped product manifolds are actually the total spaces of Riemannian submersions with
 vanishing $A$-tensor. Moreover, the stability of Einstein metrics on some interesting noncompact warped product manifolds
 with $\mathbb{R}$ as the base has been studied in \cite{Kro17} and \cite{Kro18}.

\medskip

Next, we consider the case when the base manifold $(B^{n}, \check{g})$ is a Riemannian product $(B^{n_{1}}_{1}, \check{g}_{1})\times\cdots\times(B^{n_{m}}_{m}, \check{g}_{m}).$  Let $\|A^{(p)}\|$ be the norm of the restriction of the $A$-tensor to the part of the horizontal distribution lifted from the $p$-th base factor.  We then examine the stability of the Einstein metric $g$ along  certain TT-tensors lifted from the base.

\begin{thm}\label{StabilityFromBase}
For any $1\leq p\neq q\leq m$, we have
\begin{eqnarray*}
\hspace{-3cm}\lefteqn{\left\langle(\nabla^{*}\nabla-2\mathring{R})\pi^{*}\left(\frac{1}{n_{p}}\check{g}_{p}-\frac{1}{n_{q}}\check{g}_{q}\right), \,\pi^{*}\left(\frac{1}{n_{p}}\check{g}_{p}-\frac{1}{n_{q}}\check{g}_{q}\right)\right\rangle = }  \\
\hspace{-3cm}  &   &  \hspace{4cm} -\,\frac{2}{n^{2}_{p}}s_{\check{g}_{p}}-\,\frac{2}{n^{2}_{q}}s_{\check{g}_{q}}
+ \,\frac{8}{n^{2}_{p}}\|A^{(p)}\|^{2}+\,\frac{8}{n^{2}_{q}}\|A^{(q)}\|^{2}.
\end{eqnarray*}
\end{thm}

We can apply Theorem $\ref{StabilityFromBase}$ to deduce the instability of Einstein metrics on fiber bundles constructed
in \cite{WZ90} and \cite{Wan92}.


\begin{cor} $($cf \cite{Wan16}, Chapter 4$)$ \label{TorusBundleInstability}
The Einstein metrics on the  $T^{r}$ bundles over products of $m$ Fano K\"{a}hler Einstein manifolds constructed in \cite{WZ90} have
coindex at least $m-r$, i.e., there is an $(m-r)$-dimensional subspace of TT-tensors on which
$\langle \nabla^* \nabla h - \mathring{R} h, h \rangle_{L^2(M)} < 0.$
\end{cor}

In particular, in the circle bundle case, i.e., $r=1$, Corollary \ref{TorusBundleInstability} sharpens the corresponding result in \cite{Boh05}, which gives a lower estimate for the {\em sum} of the nullity and coindex at the Einstein metrics.
Corollary 6 in \cite{WZ90} shows that if the topology of the circle bundle is sufficiently complicated, any non-trivial Einstein deformation of a base K\"ahler Einstein factor gives rise to a non-trivial Einstein deformation of the Einstein metric $g$ on the circle bundle. It produces a $TT$-tensor in the null space of the operator $\nabla^{*}\nabla-2\mathring{R}$. Thus some of the Einstein metrics on circle bundles constructed in \cite{WZ90} have nullity already greater than $(m-1)$.
Our analysis of their instability here uses only the Riemannian submersion structure and Einstein equations, and avoids
comparison with the homogeneous subcases and appealing to the variational theory for homogeneous Einstein metrics
developed in \cite{BWZ04} and \cite{Boh04}. Our approach also yields, without appealing to the fundamental conjecture
about positive quaternionic K\"ahler manifolds, the following

\begin{cor}\label{QuaternionicInstability}
The Einstein metrics constructed in \cite{Wan92} on the fiber bundles $($with fiber
$(\SO(3) \times \cdots \times \SO(3))/\Delta \SO(3)$ $)$ over a product of $m$ quaternionic K\"{a}hler
manifolds of positive scalar curvature have coindex at least $m-1$.
\end{cor}

The proofs of the above results are given in \S 3, together with some remarks regarding their relevance
to constructing ancient solutions of the Ricci flow.

\medskip

In the final section of this paper we study further the case of Riemannian submersion type Einstein metrics on the total
spaces of circle bundles. A particularly interesting and important special case is that of regular Sasaki-Einstein metrics,
which exist on certain special circle bundles over K\"ahler Einstein manifolds. By combining Corollary 6.1 of \cite{Wan17}
and an observation after Example 2.3 in \cite{CHI04}, we show that if the second Betti number $b_{2}$ of the K\"{a}hler
Einstein base is strictly greater than $1$ then the regular Sasaki-Einstein manifold is unstable. This
in particular includes the instability result in Corollary 6.2 of \cite{Wan17}. Thus for the stability of regular
Sasaki-Einstein manifolds, we are reduced to those spaces with Fano K\"{a}hler Einstein bases
with $b_{2}=1$. In a separate article we shall prove the instability of regular Sasaki-Einstein manifolds over
certain irreducible hermitian symmetric spaces and for the Aloff-Wallach spaces.

With regard to the circle bundles in Corollary $\ref{TorusBundleInstability}$ (and also the torus bundles in
Corollary \ref{TorusBundleInstability}), while the base manifolds are the product of Fano K\"{a}hler Einstein
manifolds, the submersed metrics on them are in most cases not Einstein. Indeed, by Theorem 9.76 in \cite{Bes87},
if the base manifold is Einstein in an Einstein principal circle bundle with positive scalar curvature,
the base metric has a compatible almost K\"{a}hler structure.  By the resolution of the Goldberg conjecture
\cite{Sek87}, the base metric is then actually K\"{a}hler Einstein.

It is therefore natural to consider Einstein circle bundles with non-Einstein bases. More precisely,
let $\pi: (M^{n+1}, g)\rightarrow (B^{n}, \check{g})$ be a principal circle bundle with principal connection $\theta$
such that $g(X, Y)=(\pi^{*}\check{g})(X, Y)+\theta(X)\,\theta(Y)$ for any pair of vector fields $X, Y$ on $M$.
Let $\Omega=d\theta$ be the curvature form of the connection $\theta$. Then there exists a closed $2$-form $\omega$ on $B$
such  that $\Omega=\pi^{*}\omega$. Now consider $\check{h}\in C^{\infty}(B, S^{2}(T^{*}B))$ defined by
\begin{equation}
\check{h}(\check{X}, \check{Y})=\sum^{n}_{i=1}\omega(\check{X}, \check{X}_{i})\,\omega(\check{Y}, \check{X}_{i}),
\end{equation}
where $\{\check{X}_{1}, \cdots, \check{X}_{n}\}$ is a local orthonormal basis on $B$.

\begin{thm}\label{CircleBundleInstability}
In the above situation, if $(M^{n+1}, g)$ is Einstein  but $(B^{n}, \check{g})$ is not Einstein, then
$\check{h}-\frac{\|\omega\|^{2}}{n}\check{g}$ is a $TT$-tensor on $B$
whose pull-back is a TT-tensor on $M$. Moreover, we have the formula
\begin{eqnarray*}
\lefteqn{ \left\langle(\nabla^{*}\nabla-2\mathring{R})\pi^{*}\left(\check{h}-\frac{\|\omega\|^{2}}{n}\check{g}\right), \pi^{*}\left(\check{h}-\frac{\|\omega\|^{2}}{n}\check{g}\right)\right\rangle}  \\
& =  &\{2\langle\delta^{\check{\nabla}}d^{\check{\nabla}}\check{h}, \check{h}\rangle-\langle\check{\nabla}^{*}\check{\nabla}\check{h}, \check{h}\rangle+\tr_{\check{g}}(\check{h}\circ \check{h}\circ \check{h})
-\frac{1}{2}\,\|\omega\|^{2}\|\check{h}\|^{2}-\frac{2}{n}\, \|\omega\|^{2}\|\check{h}\|^{2}  \\
&  & +\, \frac{1}{2n}\|\omega\|^{6}+ \,\frac{1}{n^{2}}\|\omega\|^{6}\}\circ\pi
\end{eqnarray*}
\noindent{where} $d^{\check{\nabla}}$ denotes the exterior differential operator acting on differential forms on $B$ with values
in the cotangent bundle $T^{*}B$ $($equipped with the Levi-Civita connection $\check{\nabla}$ of $\check{g}$ $)$,
and $\delta^{\check{\nabla}}$ denotes its formal adjoint operator.
\end{thm}

Recall that a Riemannian metric has {\em harmonic curvature} if its Levi-Civita connection is a Yang-Mills connection
for the associated orthogonal frame bundle of the tangent bundle. Equivalently, a metric $\check{g}$, e.g.,
the metric on the base in a Riemannian submersion,  has harmonic curvature if and only if its Ricci tensor $\Ric_{\check{g}}$
is a Codazzi tensor, i.e., $d^{\check{\nabla}} \Ric_{\check{g}} = 0$. In our situation, the Einstein condition on the base
implies that harmonic curvature for $\check{g}$ is equivalent to $h$ being a Codazzi tensor, i.e., $d^{\check{\nabla}}\check{h}=0$.
A special case of this condition occurs when $\omega$ is parallel, so that $\check{h}$ and therefore $\Ric_{\check{g}}$
are parallel. Then $\check{g}$ would be a local product of Einstein metrics as in the previous corollaries. However,
the class of manifolds with harmonic curvature is strictly larger than the class of local products of
Einstein manifolds, see Chapter 16 of \cite{Bes87}

For this class we obtain the following instability result.

\begin{cor}\label{HarmonicCurvatureBaseInstability}
If the base metric $\check{g}$ has harmonic curvature but is not Einstein, then the Einstein metric $g$
on the circle bundle is unstable.
\end{cor}

In \S 4, by sharpening the analysis for the case where $\check{g}$ is Einstein,
we obtain the following rigidity statement:

\begin{cor}\label{StableCircleBundleRigidity}
If the Einstein metric $g$ on the circle bundle is stable and $\check{g}$ has harmonic curvature, then the base
$(B, \check{g})$ must be a Fano K\"{a}hler Einstein manifold with $b_{2}=1$.
\end{cor}


\section{\bf Bundle analog of product type instability}

In this section, we will prove Theorem $\ref{ProductTypeInstability}$ and give some applications.

Let $\pi: (M^{n+r}, g)\rightarrow (B^{n}, \check{g})$ be a Riemannian submersion with totally geodesic fibers and $\hat{g}_{b}$ denote the induced metric on the fiber $F_{b}=\pi^{-1}(b)$ for $b\in B$. In the following, let $\{X_{1}, \cdots, X_{n}, V_{1}, \cdots, V_{r}\}$ be a local orthonormal basis in the neighborhood of a fixed but arbitrary point in $M$, where $X_{1}, \cdots, X_{n}$ are basic vector fields, namely horizontal and projectable vector fields, and $V_{1}, \cdots, V_{r}$ are vertical vector fields. Our convention for norms
is to take the full tensor norm, even for totally antisymmetric or symmetric tensors. We will also use the setup and notation
in \cite{Bes87}.

Recall from Proposition 9.61 in \cite{Bes87} that $g$ is an Einstein metric with Einstein constant $E$ iff
\begin{equation}
\check{\delta}A=0,
\end{equation}
\begin{equation}
\Ric_{\hat{g}}(U, V)+(AU, AV)=E g(U, V), \ \  \text{for vertical vector fields} \ \  U \ \  \text{and} \ \ V,
\end{equation}
\begin{equation}
(\pi^{*}\Ric_{\check{g}})(X, Y)-2(A_{X}, A_{Y})=E g(X, Y)  \ \  \text{for horizontal vector fields} \ \  X \ \  \text{and} \ \ Y,
\end{equation}
where $A$ is the O'Neill's $A$-tensor for Riemannian submersions, $(AU, AV):=\sum\limits^{n}_{i=1} \,g(A_{X_{i}}U, A_{X_{i}}V)$ and $(A_{X}, A_{Y}):=\sum\limits^{n}_{i=1} \,g(A_{X}X_{i}, A_{Y}X_{i})$.

If we denote the Einstein constant of $g$ by $E$, the above equations immediately imply that the scalar curvature $\hat{s}$ of $\hat{g}$ and the norm $\|A\|$ of the $A$-tensor are constant on $M$, and that  the scalar curvature $\check{s}$ of $\check{g}$ is constant on $B$. Moreover,
\begin{equation}\label{ATensorEq1}
\hat{s}+\|A\|^{2}=Er,
\end{equation}
and
\begin{equation}\label{ATensorEq2}
\check{s}-2\|A\|^{2}=En,
\end{equation}
where $\|A\|^{2}=\sum\limits^{n}_{i=1}(A_{X_{i}}, A_{X_{i}})=\sum\limits^{r}_{j=1}(AV_{j}, AV_{j})$.

\begin{lem}\label{DivergenceFree}
$h=\frac{1}{r}g-\frac{n+r}{nr}\pi^{*}\check{g}$ is a TT-tensor on $M$.
\end{lem}
\begin{proof}
Clearly, $\tr_{g}h=0$ and $\delta_{g}g=0$. So it suffices to check that $\delta_{g}(\pi^{*}\check{g})=0$. We show this for horizontal and vertical vector fields separately.

First, for a basic vector field $X$
\begin{eqnarray*}
\lefteqn{\delta_{g}(\pi^{*}\check{g})(X)} \\
&=&-\sum^{n}_{i=1}(\nabla_{X_{i}}\pi^{*}\check{g})(X_{i}, X)-\sum^{r}_{j=1}(\nabla_{V_{j}}\pi^{*}\check{g})(V_{j}, X)\\
&=&-\sum^{n}_{i=1}[X_{i}(\check{g}(\pi_{*}X_{i}, \pi_{*}X)\circ\pi)-\check{g}(\pi_{*}(\nabla_{X_{i}}X_{i}), \pi_{*}X)\circ\pi-\check{g}(\pi_{*}X_{i}, \pi_{*}(\nabla_{X_{i}}X))\circ\pi]\\
& & -\sum^{r}_{j=1}[V_{j}(\check{g}(\pi_{*}V_{j}, \pi_{*}X)\circ\pi)-\check{g}(\pi_{*}(\nabla_{V_{j}}V_{j}), \pi_{*}X)\circ\pi-\check{g}(\pi_{*}V_{j}, \pi_{*}(\nabla_{V_{j}}X))\circ\pi]\\
&=&-\sum^{n}_{i=1}[(\pi_{*}X_{i})(\check{g}(\pi_{*}X_{i}, \pi_{*}X))-\check{g}(\nabla^{\check{g}}_{\pi_{*}X_{i}}\pi_{*}X_{i}, \pi_{X})-\check{g}(\pi_{*}X_{i}, \nabla^{\check{g}}_{\pi_{*}X_{i}}\pi_{*}X)]\circ\pi\\
&=&(\delta_{\check{g}}\check{g})(\pi_{*}X)\circ\pi = 0.
\end{eqnarray*}

Note that in the above we have used the assumption that the $T$-tensor vanishes, which implies $\pi_{*}(\nabla_{V_{j}}V_{j})=0$.

Next, for a vertical vector field $V$ we have
\begin{eqnarray*}
 \delta_{g}(\pi^{*}\check{g})(V)
       & = &-\sum^{n}_{i=1}(\nabla_{X_{i}}\pi^{*}\check{g})(X_{i}, V)-\sum^{r}_{j=1}(\nabla_{V_{j}}\pi^{*}\check{g})(V_{j}, V)\\
&=&-\sum^{n}_{i=1}[X_{i}((\pi^{*}\check{g})(X_{i}, V))-(\pi^{*}\check{g})(\nabla_{X_{i}}X_{i}, V)-(\pi^{*}\check{g})(X_{i}, \nabla_{X_{i}}V)]\\
&  & -\sum^{r}_{j=1}[V_{j}((\pi^{*}\check{g})(V_{j}, V))-(\pi^{*}\check{g})(\nabla_{V_{j}}V_{j}, V)-(\pi^{*}\check{g})(V_{j}, \nabla_{V_{j}}V)]\\
&=&\sum^{n}_{i=1}g(X_{i}, \nabla_{X_{i}}V)\\
&=&-\sum^{n}_{i=1}g(A_{X_{i}}X_{i}, V) = 0 \\
\end{eqnarray*}
since $A_{X}Y=\frac{1}{2}\mathscr{V}[X, Y]$ for any horizontal vector fields $X$ and $Y$, where $\mathscr{V}$
denotes projection onto the vertical subspace.
\end{proof}

\begin{lem}\label{LaplaceTermI}
\begin{equation}
\langle \nabla(\pi^{*}\check{g}), \nabla(\pi^{*}\check{g}) \rangle=2\|A\|^{2}.
\end{equation}
\end{lem}
\begin{proof}
\begin{eqnarray*}
\lefteqn{ \langle \nabla\pi^{*}\check{g}, \nabla\pi^{*}\check{g}\rangle } \\
&=&\sum^{n}_{i, k, l=1}[(\nabla_{X_{i}}\pi^{*}\check{g})(X_{k}, X_{l})]^{2}+2\sum^{n}_{i, k=1}\sum^{r}_{l=1}[(\nabla_{X_{i}}\pi^{*}\check{g})(X_{k}, V_{l})]^{2}+\sum^{n}_{i=1}\sum^{r}_{k, l=1}[(\nabla_{X_{i}}\pi^{*}\check{g})(V_{k}, V_{l})]^{2}\\
& & +\sum^{r}_{i=1}\sum^{n}_{k, l=1}[(\nabla_{V_{i}}\pi^{*}\check{g})(X_{k}, X_{l})]^{2}+2\sum^{r}_{i, l=1}\sum^{n}_{k=1}[(\nabla_{V_{i}}\pi^{*}\check{g})(X_{k}, V_{l})]^{2}+\sum^{r}_{i, k, l=1}[(\nabla_{V_{i}}\pi^{*}\check{g})(V_{k}, V_{l})]^{2}\\
&=&2\sum^{n}_{i, k=1}\sum^{r}_{l=1}[-\pi^{*}\check{g}(X_{k}, \nabla_{X_{i}}V_{l})]^{2}+\sum^{r}_{i=1}\sum^{n}_{k, l=1}[V_{i}(\pi^{*}\check{g}(X_{k}, X_{l}))-\pi^{*}\check{g}(\nabla_{V_{i}}X_{k}, X_{l}) -\pi^{*}\check{g}(X_{k}, \nabla_{V_{i}}X_{l})]^{2}\\
&=&2\sum^{n}_{i, k=1}\sum^{r}_{l=1}[-g(X_{k}, \nabla_{X_{i}}V_{l})]^{2}+\sum^{r}_{i=1}\sum^{n}_{k, l=1}[-g(\nabla_{V_{i}}X_{k}, X_{l})-g(X_{k}, \nabla_{V_{i}}X_{l})]^{2}\\
&=&2\sum^{n}_{i, k=1}\sum^{r}_{l=1}[g(\nabla_{X_{i}}X_{k}, V_{l})]^{2}+\sum^{r}_{i=1}\sum^{n}_{k, l=1}[-V_{j}(g(X_{k}, X_{l}))]^{2}\\
&=&2\sum^{n}_{i, k=1}\sum^{r}_{l=1}[g(A_{X_{i}}X_{k}, V_{l})]^{2}\\
&=&2\sum^{n}_{i, k=1}g(A_{X_{i}}X_{k}, A_{X_{i}}X_{k})=2\|A\|^{2}
\end{eqnarray*}
\end{proof}

Next, we recall that, with the curvature convention $R_{X, Y} = \nabla_{[X, Y]} - [\nabla_X, \nabla_Y]$,
the action of curvature on symmetric $2$-tensors is given by $\mathring{R}(h)_{ij} := R_{ikjl} h_{kl}$.
Then by (9.28 f) in \cite{Bes87}, we easily deduce
\begin{lem}\label{CurvatureTerm}
\begin{equation}
\langle\mathring{R}\pi^{*}\check{g}, \pi^{*}\check{g}\rangle=\check{s}-3\|A\|^{2}.
\end{equation}
\end{lem}

\medskip

We may now proceed with the

\medskip

\noindent{\em Proof of Theorem $\ref{ProductTypeInstability}$}:

\medskip

\begin{eqnarray*}
\lefteqn{ \int_{M} \left\langle \left(\nabla^{*}\nabla-2\mathring{R}\right)\left(g-\frac{n+r}{n}\pi^{*}\check{g}\right), \left(g-\frac{n+r}{n}\pi^{*}\check{g}\right) \right\rangle d\vol_{g}  } \\
& = & \int_{M}  \left[-2s + 4(n+r) E +\left(\frac{n+r}{n}\right)^{2}\langle\nabla\pi^{*}\check{g}, \nabla\pi^{*}\check{g}\rangle-2\left(\frac{n+r}{n}\right)^{2}\langle\mathring{R}\pi^{*}\check{g}, \pi^{*}\check{g}\rangle \right]d\vol_{g}.
\end{eqnarray*}

Substituting the identities in Lemmas \ref{LaplaceTermI} and \ref{CurvatureTerm} into the above, we obtain

$$  \int_{M} \left[-2s+4(n+r) E +  \left(\frac{n+r}{n}\right)^{2} 2\|A\|^{2}-2\left(\frac{n+r}{n}\right)^{2}(\check{s}-3\|A\|^{2})\right] d\vol_{g}. $$
If we simplify the integrand using equations (\ref{ATensorEq1}) and $(\ref{ATensorEq2})$, then we obtain
upon integration

$$ \Vol(M, g) \left(\frac{2(n+r)}{n^{2}}\right)(r\check{s}-2n\hat{s}) $$
as claimed.

\begin{rmk}  \label{can-var}
In the situation of Theorem \ref{ProductTypeInstability}, if the fiber and base metrics are in addition
Einstein with {\em positive} Einstein constants $\hat{E}$ and $\check{E}$ respectively, then
$$ r\check{s}-2n\hat{s} = rn (\check{E} -2 \hat{E}). $$
So instability of the Einstein metric $g$  is implied by $\check{E} < 2 \hat{E}$. Note that we are in
particular in the situation of the {\em canonical variation} discussed on pp. 253-255 of \cite{Bes87}.
There the auxiliary function $\varphi$ is just the restriction of the functional $\widetilde{\bf S}$
to the metrics lying in the canonical variation, divided by the volume of $g$, which is a constant.
The first graph of $\varphi$ on p. 254 shows that one of the Einstein metrics is always unstable
in the sense of not being a local maximum of the normalized scalar curvature functional along TT-tensor directions.
Thus Theorem \ref{ProductTypeInstability} may be thought of as a generalization of this study to the
more general situation in which the fibers and base are not necessarily Einstein. It also exhibits
an explicit destabilizing TT-tensor.
\end{rmk}

We next give an example to show how Theorem \ref{ProductTypeInstability} can be applied.

\begin{example} We consider the normal homogeneous metric $g$ on $M= \Sp(mq)/(\Sp(q) \times \cdots \times \Sp(q))$
(where there are $m$ factors in the denominator) induced by the negative of the Killing form of $\Sp(mq)$. We shall assume
that $q \geq 1$ and $m \geq 3$. It is known that $g$ is Einstein (\cite{WZ85}, Table IA). Choose an integer $k$ such that
$1 < k < m$. Then we can fiber $M$ over $B_k = \Sp(mq)/(\Sp(kq) \times \Sp(q) \times \cdots \times \Sp(q))$
with fiber $F_k = \Sp(kq)/(\Sp(q) \times \cdots \times \Sp(q))$. This is a Riemannian submersion with totally
geodesic fibers if we equip the fiber and base also by the normal metrics $\hat{g}$ and $\check{g}$ induced by
the negative of the Killing form of $\Sp(mq)$. Notice that $\hat{g}$ is Einstein but $\check{g}$ is not Einstein unless
$k=m-1$, when the base is a symmetric space. Each of the above Riemannian submersions gives rise to a TT-tensor by
Lemma \ref{DivergenceFree}.
To compute the quantity $r\check{s}-2n\hat{s}$ we will use the results in sections 1.2, 2.1 and 2.2 of
\cite{WZ85}.

The dimension $r$ of the fiber is $2k(k-1)q^2$ and the dimension $n$ of the base is $2(m-k)(m+k-1)q^2$.
The fiber is Einstein because the Killing form of $\Sp(mq)$ restricts to $\frac{mq+1}{kq+1}$ times
the Killing form of $\Sp(kq)$ and the irreducible summands in the isotropy representation of the fiber
are all equivalent under outer automorphisms of $\Sp(kq)$  (see Corollary 1.14 in \cite{WZ85}). The
Casimir constants of the irreducible summands with respect to the Killing form of $\Sp(kq)$ are
$\frac{2q+1}{2(kq+1)}$. Hence by (1.7) and (1.10) in \cite{WZ85} the Einstein constant of the metric
$\hat{g}$ of the fiber is given by
$$\hat{\lambda} = \frac{1}{4} \left(1 + \frac{2q+1}{kq+1} \right) \left( \frac{kq+1}{mq+1} \right),$$
and the scalar curvature $\hat{s} = \frac{k(k-1)q^2}{2(mq+1)} (qk + 2q +2)$.

The irreducible summands of the isotropy representation of the base consist of $m-k$ summands of
the form $\nu_{kq} {\otimes} \cdots \otimes \nu_q \otimes \cdots$ (where $\nu_{\ell}$ denotes the
$2\ell$-dimensional symplectic representation of $\Sp(\ell)$ on $\C^{2\ell}$ and $\cdots$ denotes
trivial representations of the corresponding factors) and $\frac{1}{2}(m-k)(m-k-1)$ summands of the form
$\cdots \otimes \nu_q \otimes \cdots \otimes \nu_q \otimes \cdots$. The respective Casimir constants with respect to
the Killing form of $\Sp(mq)$ are
$$ \frac{(k+1)q +1}{2(mq+1)}  \, \,\,\, \mbox{\rm and} \,\,\,\, \frac{2q+1}{2(mq+1)}.$$
So the corresponding eigenvalues of the Ricci tensor of $\check{g}$ are
$$ \frac{1}{4} \left(1 + \frac{kq +q +1}{mq +1} \right) \,\,\,\,\mbox{\rm and} \,\,\,\,
      \frac{1}{4}\left(1 + \frac{2q+1}{mq+1}   \right).$$
We obtain
$$ \check{s} = \frac{q^2 (m-k)}{mq+1} \left((m+k+1)kq + 2k + \frac{m-k-1}{2}\,\,(mq +2q +2) \right).$$
Therefore, the quantity
\medskip

\noindent{$  r\check{s}-2n\hat{s} \,\,= \frac{2k(k-1)(m-k)q^4}{mq +1} \,\, \cdot $}

$ \hspace{1cm}   \left[ (m+k+1)kq + 2k + \frac{m-k-1}{2} \, (mq +2q+2) - (m+k-1)(kq +2q+2) \right]. $

\medskip
\noindent{The factor in} square brackets in the above simplifies to
$$ \frac{1}{2} \left( qm(m-k-3) -2q(k-1) -2m -2(k-1)       \right). $$
If this quantity is negative, then the TT-tensor determined by the Riemannian
submersion gives an unstable variation of $g$ for the functional $\widetilde{\bf S}$.
This is the case in particular if $m-3 \leq k$. However, when $k = m-4$, then
$r\check{s}-2n\hat{s}$ is negative iff we further have $\frac{10(q+1)}{q+4} < m.$

Analogous results hold for the normal homogeneous Einstein spaces
$\SU(mq)/S(\U(q) \times \cdots \times \U(q)), q \geq 2, m \geq 3$
and $\SO(mq)/(\SO(q) \times \cdots \times \SO(q)), q \geq 4, m \geq 3.$ More precisely, for $k \geq 2$,
the projections onto $\SU(mq)/S(\U(kq) \times \U(q) \times \cdots \times \U(q))$ with $m-3 \leq k$
and onto $\SO(mq)/( \SO(kq) \times \SO(q) \times \cdots \times \SO(q))$ with $m-2 \leq k$
give rise to destablizing TT-tensors.
\end{example}


\section{\bf Instability from product structure on the base}

In this section, we will prove Theorem $\ref{StabilityFromBase}$ and its corollaries.

Let $\pi: (M^{n+r}, g)\rightarrow (B^{n}, \check{g})$ be a Riemannian submersion with totally geodesic fibers
where $g$ is Einstein and the base splits as a Riemannian product
$$ (B^{n}, \check{g})=(B^{n_{1}}_{1}, \check{g}_{1})\times\cdots\times(B^{n_{m}}_{m}, \check{g}_{m}). $$
For each $1\leq p\leq m$, let $\{\check{X}^{(p)}_{1}, \cdots, \check{X}^{(p)}_{n_{p}}\}$ be a local orthonormal basis
on $(B^{n_{p}}_{p}, \check{g}_{p})$ and  $\{X^{(p)}_{1}, \cdots, X^{(p)}_{n_{p}}\}$ its horizontal lift to $M$.
Together with vertical vector fields $V_{1}, \cdots, V_{r}$,
$\{X^{(1)}_{1}, \cdots, X^{(1)}_{n_{1}}, \cdots, X^{(m)}_{1}, \cdots, X^{(m)}_{n_{m}}, V_{1}, \cdots, V_{r}\}$
becomes a local orthonormal basis on $(M^{n+r}, g)$.

Analogous computations as those in Lemmas $\ref{DivergenceFree}$, $\ref{LaplaceTermI}$, and
$\ref{CurvatureTerm}$ lead to the following:

\begin{lem}  \label{productbase-TT}
For any $1\leq p, q \leq m$, $\pi^{*}\left(\frac{1}{n_{p}}\check{g}_{p}-\frac{1}{n_{q}}\check{g}_{q}\right)$ is a TT-tensor on $(M, g)$.
\end{lem}

\begin{lem}  \label{productbase-lemma}
For any $1\leq p\neq q\leq m$,
\begin{equation}
\langle\nabla\pi^{*}\check{g}_{p}, \nabla\pi^{*}\check{g}_{p}\rangle=2\|A^{(p)}\|^{2},
\end{equation}

\begin{equation}
\langle\nabla\pi^{*}\check{g}_{p}, \nabla\pi^{*}\check{g}_{q}\rangle=0,
\end{equation}

\begin{equation}
\langle \mathring{R}(\pi^* \check{g}_p), \pi^* \check{g}_q \rangle = 0, \,\,
   \langle \mathring{R}(\pi^* \check{g}_p), \pi^* \check{g}_p \rangle = s_{\check{g}_p} - 3 \|A^{(p)}\|^2,
\end{equation}

\begin{equation}
\left\langle\mathring{R}\pi^{*}\left(\frac{1}{n_{p}}\check{g}_{p}-\frac{1}{n_{q}}\check{g}_{q}\right), \left(\frac{1}{n_{p}}\check{g}_{p}-\frac{1}{n_{q}}\check{g}_{q}\right)\right\rangle =
\frac{s_{\check{g}_{p}}}{n^{2}_{p}}+\frac{s_{\check{g}_{q}}}{n^{2}_{q}}
-\frac{3}{n^{2}_{p}}\|A^{(p)}\|^{2}-\frac{3}{n^{2}_{q}} \|A^{(q)}\|^{2},
\end{equation}
where $\|A^{(p)}\|^{2} :=\sum^{n_{p}}_{i, j=1}g(A_{X^{(p)}_{i}}X^{(p)}_{j}, A_{X^{(p)}_{i}}X^{(p)}_{j}).$
\end{lem}

The Riemannian product structure enters in the computations through the vanishing of covariant
derivatives and brackets (including the $A$-tensor) involving vectors $\check{X}^{(p)}_{i}$ and $\check{X}^{(q)}_{j}$ with $p \neq q$.
The above lemmas in turn imply Theorem $\ref{StabilityFromBase}$ in a straightforward manner.

Now we apply Theorem $\ref{StabilityFromBase}$ to some examples.

\subsection{Proof of Corollary \ref{TorusBundleInstability}}
In this case let $(B^{n_{i}}_{i}, g_{i})$, $i=1, \cdots, m$, be K\"{a}hler Einstein manifolds with $c_{1}(B_{i})>0$ and real dimension $n_{i}$. Then one can write $c_{1}(B_{i})=q_{i}\alpha_{i}$ with $\alpha_{i}$ an indivisible integral cohomology class and $q_i > 0$.
As in \cite{WZ90} we normalize $g_{i}$ such that the de Rham class $[\omega_{i}]=2\pi \alpha_{i}$, where $\omega_{i}$ is the K\"{a}hler form of $g_{i}$. This is equivalent to the condition $\Ric_{g_{i}}=q_{i}g_{i}$.

Consider a principal torus $T^{r}$ bundle $\pi: M\rightarrow B=B_{1}\times \cdots \times B_{m}$ that is classified by $r$ integral cohomology classes
$$ \chi_{\beta} = \sum_{i=1}^m \, b_{\beta i} \pi^{*}_{i}\alpha_i, \,\,\,\, 1 \leq \beta \leq r,$$
where $\pi_{i}: B\rightarrow B_{i}$ is the projection onto the $i$-th factor. (Note that this presumes a fixed splitting of the torus as a product of circles.)
We will essentially use the notation in section 1 of \cite{WZ90} except that we will use the full tensor norm for $2$-forms, as we have already indicated before, and use $(\hat{g}_{\alpha \beta})$ to denote
the metric on the torus fibers. In particular the $r \times m$ matrix $(b_{\alpha i})$ has full rank with $r \leq m$.
If $r=m\geq2$, the Einstein metric constructed on $M$ is actually a product Einstein metric, which is unstable.
Also, if the $k$-th column of $(b_{\alpha i})$ is zero, then $M$ is a Riemannian product of a torus bundle
over the product of the remaining base factors and the $k$-th K\"ahler Einstein factor, and hence is again
unstable. We shall therefore assume that all columns in $(b_{\alpha i})$ are nonzero.

It will be convenient to denote the $j$-th column of $(b_{\alpha i})$ by $C_j$ and the inner products
between the columns with respect to the fiber metric by $C_{jk}$, i.e.,
$$  C_{jk} := \sum_{\alpha, \beta} \, b_{\alpha j} \hat{g}_{\alpha \beta} b_{\beta k}.$$
Notice that $C_{jj}$ are positive as $C_j$ are nonzero vectors by assumption. Then by possibly permuting
the K\"ahler Einstein factors we may further assume that
\begin{equation} \label{orderednorms}
 \frac{n_1}{x_1^2}\, C_{11} \leq \cdots \leq \frac{n_m}{x_m^2}\, C_{mm}.
\end{equation}

Let $\check{g}_i = x_i g_i$ with $x_i > 0$, so that its scalar curvature $s_{\check{g}_i} = \frac{n_i q_i}{x_i}$. Recall that the Riemannian submersion structure of the Einstein metric $g$ is  given by a product
metric $\check{g} = \check{g}_1 \oplus \cdots \oplus \check{g}_m$ on the base,
and an $\R^r$-valued connection $\theta$ with $\check{g}$-harmonic curvature form
$\Omega = \sum_{\alpha=1}^r \, \sum_{i=1}^m \, b_{\alpha i} \pi_i^*\omega_i e_{\alpha}$.
The harmonicity condition is independent of the choices $x_i$.
It then follows from the relation $\theta(A_X Y) = -\frac{1}{2} \Omega(X, Y)$ that
$$ \frac{1}{4} \frac{n_i}{x_i^2} C_{ii} =  \sum_{j, k = n_{i-1}+ 1}^{n_{i-1} + n_i} \, g(A_{X_j}X_k, A_{X_j}X_k)
        =  \|A^{(i)} \|^2.$$

Let $ h_i :=\pi^*( \frac{\check{g}_i}{n_{i}} - \frac{\check{g}_{i+1} }{n_{i+1}}), i=1, \cdots, m-r$.
These $2$-tensors are linearly independent, and they span an $(m-r)$-dimensional subspace of the TT-tensors by Lemma \ref{productbase-TT}.
Consider $h = \mu_1 h_1 + \cdots + \mu_{m-r} h_{m-r}$ where $\mu_i \in \R$. Writing $h$ in the form
$\sum^{m-r+1}_{i=1}(\mu_{i}-\mu_{i-1})\pi^{*}\left(\frac{\check{g}_{i}}{n_{i}}\right)$ with $\mu_{0}=\mu_{m-r+1}=0$, and applying Lemma \ref{productbase-lemma} and equations (1.11)-(1.12) of \cite{WZ90}, we obtain
\begin{eqnarray*}
\langle\nabla h, \nabla h\rangle-2\langle\mathring{R}h, h\rangle
&=& -\sum^{m-r+1}_{i=1}(\mu_{i}-\mu_{i-1})^{2}\frac{2q_{i}}{n_{i}x_{i}}
+\sum^{m-r+1}_{i=1}(\mu_{i}-\mu_{i-1})^{2}\frac{2C_{ii}}{n_{i}x^{2}_{i}}\\
&=&-\sum^{m-r+1}_{i=1}(\mu_{i}-\mu_{i-1})^{2}\frac{2}{n_{i}}\left(E+\frac{C_{ii}}{2x^{2}_{i}}\right)
+\sum^{m-r+1}_{i=1}(\mu_{i}-\mu_{i-1})^{2}\frac{2C_{ii}}{n_{i}x^{2}_{i}}\\
&=& -\sum^{m-r+1}_{i=1}(\mu_{i}-\mu_{i-1})^{2}\frac{2}{n_{i}}E
+\sum^{m-r+1}_{i=1}(\mu_{i}-\mu_{i-1})^{2}\frac{C_{ii}}{n_{i}x^{2}_{i}}\\
&=& -\sum^{m-r+1}_{i=1}\left((\mu_{i}-\mu_{i-1})^{2}\frac{2}{n_{i}}\frac{1}{4r}\sum^{m}_{j=1}\frac{C_{jj}n_{j}}{x^{2}_{j}}\right)
+\sum^{m-r+1}_{i=1}(\mu_{i}-\mu_{i-1})^{2}\frac{C_{ii}}{n_{i}x^{2}_{i}}.
\end{eqnarray*}

For $\sum^{m}_{j=1}\frac{C_{jj}n_{j}}{x^{2}_{j}}$ in the $i$-th term in the first sum above, by applying $(\ref{orderednorms})$ we have
$$\sum^{m}_{j=1}\frac{C_{jj}n_{j}}{x^{2}_{j}}\geq\sum^{m}_{j=i}\frac{C_{jj}n_{j}}{x^{2}_{j}}\geq(m-i+1)\frac{C_{ii}n_{i}}{x^{2}_{i}}.$$
By substituting this into the first sum above, we obtain
\begin{eqnarray*}
\langle\nabla h, \nabla h\rangle-2\langle\mathring{R}h, h\rangle
&\leq& -\sum^{m-r+1}_{i=1}(\mu_{i}-\mu_{i-1})^{2} \,\frac{m-i+1}{2r}\frac{C_{ii}}{x^{2}_{i}}
+\sum^{m-r+1}_{i=1}(\mu_{i}-\mu_{i-1})^{2}\frac{C_{ii}}{n_{i}x^{2}_{i}}\\
&=& -\sum^{m-r}_{i=1}(\mu_{i}-\mu_{i-1})^{2}\frac{1}{2r}\frac{C_{ii}}{x^{2}_{i}}
-\sum^{m-r}_{i=1}(\mu_{i}-\mu_{i-1})^{2} \,\frac{m-i}{2r}\frac{C_{ii}}{x^{2}_{i}}\\
& & -\mu^{2}_{m-r} \,\frac{1}{2}\frac{C_{(m-r+1)(m-r+1)}}{x^{2}_{m-r+1}}
+\sum^{m-r+1}_{i=1}(\mu_{i}-\mu_{i-1})^{2}\frac{C_{ii}}{n_{i}x^{2}_{i}}\\
&=& -\mu^{2}_{1} \,\frac{1}{2r}\frac{C_{11}}{x^{2}_{1}}-\sum^{m-r}_{i=2}(\mu_{i}-\mu_{i-1})^{2}\frac{1}{2r}\frac{C_{ii}}{x^{2}_{i}}\\
& & -\sum^{m-r}_{i=1}(\mu_{i}-\mu_{i-1})^{2}\left(\frac{m-i}{2r}-\frac{1}{n_{i}}\right)\frac{C_{ii}}{x^{2}_{i}}\\
& & -\mu^{2}_{m-r} \,\left(\frac{1}{2}-\frac{1}{n_{m-r+1}}\right)\frac{C_{(m-r+1)(m-r+1)}}{x^{2}_{m-r+1}}\\
& \leq & -\mu^{2}_{1} \,\frac{1}{2r}\frac{C_{11}}{x^{2}_{1}}-\sum^{m-r}_{i=2}(\mu_{i}-\mu_{i-1})^{2}\frac{1}{2r}\frac{C_{ii}}{x^{2}_{i}},
\end{eqnarray*}
since $n_{i}\geq2$ and $\mu_{0}=\mu_{m-r+1}=0$. This is negative unless all the $\mu_{i}$ vanish, since all the $C_{ii}>0$. This completes the proof of Corollary \ref{TorusBundleInstability}. \hspace{9cm} $\Box$

\subsection{Proof of Corollary \ref{QuaternionicInstability}}
For this case we basically follow the notation in \cite{Wan92}. Recall that the base factors
$(B_i^{n_i}, g_i)$ are quaternionic K\"ahler manifolds with positive scalar curvature and
$\Ric_{g_i} = E_i g_i$. We will let $n_i = 4N_i$, a compromise between our notation and that
in \cite{Wan92}. The metric on the base is given by $\check{g} = x_1 g_1 \oplus \cdots \oplus
x_m g_m$ where $x_i > 0$ and $\check{g}_i = x_i g_i$. The connection defining the Riemannian submersion
structure comes from projection of the product of the $\so(3)$-valued Yang-Mills connections on the canonical
$\SO(3)$-bundle over $B_i$. The metric on the fibers $(\SO(3) \times \cdots \times \SO(3))/ \Delta \SO(3)$
is the normal metric induced from the biinvariant metric $\lambda_1 B_{\so(3)} \oplus \cdots \oplus \lambda_m B_{\so(3)}$,
where $B_{\so(3)}$ denotes the negative of the Killing form of $\so(3)$.

Using the formulas on pp. 310-311 of \cite{Wan92}, one easily derives
$$ \| A^{(i)} \|^2 = \sum_{j, k =1}^{4N_i} \, g(A_{X_j} X_k, A_{X_j} X_k)  =
    \frac{3}{2} \frac{4N_iE_i^2}{(N_i +2)^2 x_i^2} \,\lambda_i \left(1 - \frac{\lambda_i}{\lambda}\right)  $$
where $\{X_1, \cdots, X_{4N_i} \}$ is a $\check{g}_i$-orthonormal basis of basic vector fields
associated to factor $B_i$ in the base and $\lambda = \lambda_1 + \cdots + \lambda_m$.

We choose two different indices
and consider the TT-tensor $h :=  \pi^*(\frac{\check{g}_i}{4N_i} - \frac{\check{g}_j}{4N_j})$.
Without loss of generality, we may let $i=1$ and $j=2$. After applying Lemma \ref{productbase-lemma}
and some simplication we get
$$ \langle \nabla h, \nabla h \rangle - 2 \langle \mathring{R} h, h \rangle =
      - \, \frac{1}{2N_1} \frac{E_1}{x_1} - \, \frac{1}{2N_2} \frac{E_2}{x_2} + \frac{3}{N_1} \frac{E_1^2}{x_1^2 (N_1 + 2)^2}\, \lambda_1 \left(1 - \frac{\lambda_1}{\lambda} \right) $$
$$  \hspace{1cm}  +  \,\, \frac{3}{N_2} \frac{E_2^2}{x_2^2 (N_2 + 2)^2}\, \lambda_2 \left(1 - \frac{\lambda_2}{\lambda} \right). $$
Substituting equation (2.1) from \cite{Wan92} into the first two terms of the above and simplifying, we
obtain
$$ - \frac{E}{2N_1}  - \frac{E}{2N_2} + \frac{3}{2N_1} \frac{E_1^2}{(N_1 + 2)^2} \frac{\lambda_1}{x_1^2} \left(1 - \frac{\lambda_1}{\lambda} \right)
      + \frac{3}{2N_2} \frac{E_2^2}{(N_2 + 2)^2} \frac{\lambda_2}{x_2^2} \left(1 - \frac{\lambda_2}{\lambda} \right).$$

We shall now assume that $m \geq 3$ and substitute equation (2.5) of \cite{Wan92} for $E$ in the above.
This yields
 $$ - \frac{1}{2N_1} \left( \frac{1}{4\lambda_1} + \frac{1}{2\lambda} \right)    - \frac{1}{2N_2} \left( \frac{1}{4\lambda_2} + \frac{1}{2\lambda} \right)
     - \frac{E_1^2}{(N_1 + 2)^2} \frac{\lambda_1}{x_1^2} - \frac{E_2^2}{(N_2 + 2)^2} \frac{\lambda_2}{x_2^2} $$
 $$ \hspace{1cm}    + \, \frac{3}{2N_1} \frac{E_1^2}{(N_1 + 2)^2} \frac{\lambda_1}{x_1^2} \left(1 - \frac{\lambda_1}{\lambda} \right)
        + \, \frac{3}{2N_2} \frac{E_2^2}{(N_2 + 2)^2} \frac{\lambda_2}{x_2^2} \left(1 - \frac{\lambda_2}{\lambda} \right)$$
$$ = - \frac{1}{2N_1} \left( \frac{1}{4\lambda_1} + \frac{1}{2\lambda} \right)    - \frac{1}{2N_2} \left( \frac{1}{4\lambda_2} + \frac{1}{2\lambda} \right)
       - \,\frac{1}{4N_1} \frac{E_1^2}{(N_1 + 2)^2} \frac{\lambda_1}{x_1^2} \left(4N_1 -6( 1- \frac{\lambda_1}{\lambda}) \right) $$
$$ \hspace{1cm}      - \,\frac{1}{4N_2} \frac{E_2^2}{(N_2 + 2)^2} \frac{\lambda_2}{x_2^2} \left(4N_2 -6( 1- \frac{\lambda_2}{\lambda}) \right). $$
The first two terms in the above are negative, while the last two terms are non-positive since $N_i \geq 2$ for
quaternionic K\"ahler manifolds. Hence the Einstein metric $g$ is unstable.

When $m=2$, one may use equations (2.3) in \cite{Wan92} in an analogous argument to show instability.

To verify the claim about the coindex, we proceed exactly as in the proof of Corollary \ref{TorusBundleInstability}.
We consider $h = \sum_1^{m-1} \, \mu_i \left( \frac{\pi^* \check{g}_i}{4N_i} - \frac{\pi^* \check{g}_{i+1}}{4N_{i+1}}     \right)$
and repeat the above computations. One then obtains
$$ \| \nabla h \|^2 - 2 \langle \mathring{R}h, h \rangle = - \frac{\mu_1^2}{2N_1^2} \left( \frac{1}{4\lambda_1}
       - \frac{1}{2\lambda} \right)  - \sum_{i=2}^{m-1} \frac{(\mu_i - \mu_{i-1})^2}{2N_i^2} \left( \frac{1}{4\lambda_i}
       - \frac{1}{2\lambda} \right)  -  \frac{\mu_{m-1}^2}{2N_m^2} \left( \frac{1}{4\lambda_m} - \frac{1}{2\lambda} \right) $$
$$    - \frac{\mu_1^2}{2N_1} \frac{E_1^2}{(N_1 +2)^2} \frac{\lambda_1}{x_1^2} \left(2N_1 - 3 (1 - \frac{\lambda_1}{\lambda} \right)
     - \sum_{i=2}^{m-1} \frac{(\mu_i - \mu_{i-1})^2}{2N_i} \frac{E_i^2}{(N_i +2)^2} \frac{\lambda_i}{x_i^2} \left(2N_i - 3 (1 - \frac{\lambda_i}{\lambda} \right)  $$
$$   - \,  \frac{\mu_{m-1}^2}{2N_m} \frac{E_m^2}{(N_m +2)^2} \frac{\lambda_m}{x_m^2} \left(2N_m - 3 (1 - \frac{\lambda_m}{\lambda} \right).$$
Since $N_i \geq 2$ the last three terms are non-positive. The remainder of the above expression is strictly negative
unless all the $\mu_i$ vanish.  \hspace{9cm} $\Box$

\begin{rmk}  \label{4d-base}
By definition, quaternionic K\"ahler manifolds have dimension at least $8$. However, self-dual Einstein
$4$-manifolds with nonzero scalar curvature may be interpreted as $4$-dimensional analogs of quaternionic K\"ahler
manifolds. Of course, positive self-dual Einstein $4$-manifolds are allowed base factors in the bundle
construction in \cite{Wan92}. If these are present, the above computation does not yield any information.
But if there are $k$ genuine quaternionic K\"ahler factors in the base, the above analysis does
show that the  coindex of the Einstein metric on the bundle is at least $k-1$. We suspect that the coindex
lower bound remains valid even if some of the base factors are positive self-dual Einstein $4$-manifolds.
(By a celebrated theorem of Hitchin, these can only be $S^4$ and $\C\PP^2$ with their canonical metrics.)
\end{rmk}

\begin{rmk} \label{comparisons}
Note that our proofs of linear instability above use only the Einstein equations and the
Riemannian submersion structure. Unlike the analysis in \cite{Boh05} we do not use the variational theory
of compact homogeneous Einstein manifolds developed in \cite{BWZ04} and \cite{Boh04}. Our simpler approach
here produces the unstable directions directly, and, more importantly,  provides information about the coindex
rather than the sum of the nullity and coindex of the normalized Hilbert functional.

Combining our instability results with  Theorem 1.3 in \cite{Kro15}, one obtains dynamic
instability of the Einstein metric $g$, i.e., the existence of a non-trivial normalised ancient solution of the Ricci flow
whose backwards limit converges {\em modulo diffeomorphisms} to $g$. In \cite{LW17} and \cite{LW16}
$\kappa$-non-collapsed ancient solutions
of the Ricci flow with rescaled backwards limit converging to $g$ were also constructed.  The difference is that
in those works the convergence is obtained without requiring pull-backs by diffeomorphisms. This is possible
because the Ricci flow equations restricted to the class of bundle-type metrics become an ordinary differential system
for which some of the stationary points correspond to the Einstein metrics. On the other hand, the analysis in \cite{LW17}
and \cite{LW16} relies strongly on knowing the approximate location of the stationary points. This is because one needs
to know the eigenvalues and eigenvectors of the linearization of the flow system at the Einstein metrics. For the
torus bundle case and the general case with quaternionic K\"ahler base factors, the Einstein metrics are not
known explicitly enough for constructing non-collapsed ancient solutions with the Einstein metrics as backwards limit.
Because we only use the Einstein equations here and not the actual solutions, we do get ancient solutions for these cases here,
but the backwards convergence to the Einstein metric holds in a much weaker sense.
\end{rmk}


\section{\bf Instability of Einstein metrics on circle bundles}

In this section we examine more closely the stability of an Einstein principal circle bundle.
Let $\pi: (M^{n+1}, g)\rightarrow (B^{n}, \check{g})$ be a principal circle bundle with connection $\theta$, where $g(X, Y)=(\pi^{*}\check{g})(X, Y)+\theta(X)\theta(Y)$ for any pair of vector fields $X, Y$ on $M$. Then $\pi$ is a Riemannian
 submersion with totally geodesic fiber, i.e.,  O'Neill's $T$-tensor vanishes. Let $\Omega=d\theta$ be the curvature form of the connection $\theta$. Then there exists a closed $2$-form $\omega$ on $B$ such  that $\Omega=\pi^{*}\omega$. O'Neill's tensor $A$
 is related to $\Omega$ by
\begin{equation}\label{TensorAForPrincipalBundles}
\theta(A_{X}Y)=-\frac{1}{2}\,\Omega(X, Y).
\end{equation}
Then $g$ is Einstein with Einstein constant $E$ iff
\begin{equation}
\omega \ \  \text{is harmonic},
\end{equation}
\begin{equation}
\|\omega\|^{2}=4E,
\end{equation}
\begin{equation}
\Ric_{\check{g}}(\check{X}, \check{Y})-\frac{1}{2}\sum^{n}_{i=1}\omega(\check{X}, \check{X}_{i})\omega(\check{Y}, \check{X}_{i})
= E \,\check{g}(\check{X}, \check{Y}), \ \ \text{for vector fields} \ \ \check{X}, \check{Y} \ \  \text{on} \ \ B,
\end{equation}
where $\{\check{X}_{1}, \cdots, \check{X}_{n}\}$ is a local orthonormal basis of $B$.

\subsection{Relationship between the stability operators on the total space and on the base space}

Throughout this subsection, let $\{X_{1}, \cdots, X_{n}, U\}$ be a local orthonormal basis on $M$ around some fixed but
 arbitrary point in $M$, where $X_{1}, \cdots, X_{n}$ are basic vector fields whose projections $\{\check{X}_{1}=\pi_{*}X_{1}, \cdots, \check{X}_{n}=\pi_{*}X_{n}\}$ form a local orthonormal basis on $B$, and $U$ is the vertical vector field induced by the circle action on the total space $M$ with $\theta(U)=1$.

\begin{lem}
Let $X$ and $Y$ be basic vector fields. We have
\begin{eqnarray*}
[U, X] &=&\mathcal{L}_{U}X=0,\\
\nabla_{X}Y&=&\check{\nabla}_{\check{X}}\check{Y}-\frac{1}{2}\omega(\check{X}, \check{Y})U,\\
\nabla_{U}X=\nabla_{X}U&=&\frac{1}{2}\omega(\check{X}, \check{X_{i}})X_{i},\\
\nabla_{U}U&=&0.
\end{eqnarray*}
In the second equality, actually $\check{\nabla}_{\check{X}}\check{Y}$ is a vector field on the base $B$. But here we use it to denote its horizontal lift to $P$.
\end{lem}
\begin{proof}
The first equation follows from facts that $X$ is horizontal, $U$ is generated by the $S^{1}$-action, and the horizontal distribution is $S^{1}$-invariant. The rest of the equations follow from O'Neill's fundamental equations for Riemannian submersions, the vanishing of the tensor $T$, and the relation $A_{X}Y=-\frac{1}{2}\omega({\check{X}, \check{Y}})U$.
\end{proof}

Let $\check{h}\in C^{\infty}(S^{2}(B^{n}))$ be a symmetric 2-tensor on $B^{n}$,  then $h=\pi^{*}\check{h}$ is a symmetric 2-tensor on $M^{n+1}$.

\begin{lem}\label{LaplacianOnCircleBundles}
For any $1\leq i, j\leq n$
\begin{equation}
(\nabla^{*}\nabla h)_{ij}
=(\check{\nabla}^{*}\check{\nabla}\check{h})_{ij}\circ\pi+\sum^{n}_{k,l=1} \, \left(\frac{1}{2}\omega_{ki}\omega_{kl}\check{h}_{lj}+\frac{1}{2}\omega_{kj}\omega_{kl}\check{h}_{li}
-\frac{1}{2}\omega_{ik}\omega_{jl}\check{h}_{kl} \right)\circ\pi.
\end{equation}
\end{lem}
\begin{proof}
\begin{equation}\label{LaplaceTerm}
(\nabla^{*}\nabla h)_{ij}=-\sum^{n}_{k=1}(\nabla_{k}\nabla_{k}h)(X_{i}, X_{j})+\sum^{n}_{k=1}(\nabla_{\nabla_{k}X_{k}}h)(X_{i}, X_{j})-(\nabla_{U}\nabla_{U}h)(X_{i}, X_{j}),
\end{equation}
where and throughout this proof, $\nabla_{k}$ means $\nabla_{X_{k}}$, $\check{\nabla}_{k}$ means $\check{\nabla}_{\check{X_{k}}}$. Now we compute each of these three terms.
\begin{equation}\label{FirstTerm}
\begin{aligned}
(\nabla_{k}\nabla_{k}h)(X_{i}, X_{j})
&=X_{k}X_{k}(h(X_{i}, X_{j}))-X_{k}(h(\nabla_{k}X_{i}, X_{j}))-X_{k}(h(X_{i}, \nabla_{k}X_{j}))\\
&\ \ \ -X_{k}(h(\nabla_{k}X_{i}, X_{j}))+h(\nabla_{k}\nabla_{k}X_{i}, X_{j})+h(\nabla_{k}X_{i}, \nabla_{k}X_{j})\\
&\ \ \ -X_{k}(h(X_{i}, \nabla_{k}X_{j}))+h(\nabla_{k}X_{i}, \nabla_{k}X_{j})+h(X_{i}, \nabla_{k}\nabla_{k}X_{j})\\
&=[\check{X_{k}}\check{X_{k}}(\check{h}(\check{X_{i}}, \check{X_{j}}))-\check{X_{k}}(\check{h}(\check{\nabla}_{k}\check{X_{i}}, \check{X_{j}}))-\check{X_{k}}(\check{h}(\check{X_{i}}, \check{\nabla}_{k}\check{X_{j}}))\\
&\ \ \ -\check{X_{k}}(\check{h}(\check{\nabla}_{k}\check{X_{i}}, \check{X_{j}}))+\check{h}(\check{\nabla}_{k}\check{\nabla}_{k}\check{X_{i}}, \check{X_{j}})\\
&\ \ \ -\sum^{n}_{l=1}\frac{1}{4}\omega_{ki}\omega_{kl}\check{h}_{lj}+\check{h}(\check{\nabla}_{k}\check{X_{i}}, \check{\nabla}_{k}\check{X_{j}})\\
&\ \ \  -\check{X_{k}}(\check{h}(\check{X_{i}}, \check{\nabla}_{k}\check{X_{j}}))+\check{h}(\check{\nabla}_{k}\check{X_{i}}, \check{\nabla}_{k}\check{X_{j}})\\
&\ \ \ +\check{h}(\check{X_{i}}, \check{\nabla}_{k}\check{\nabla}_{k}\check{X_{j}})-\sum^{n}_{l=1}\frac{1}{4}\omega_{kj}\omega_{kl}\check{h}_{il}]\circ\pi\\
&=[(\check{\nabla}_{k}\check{\nabla}_{k}\check{h})_{ij}-\sum^{n}_{l=1}\frac{1}{4}\omega_{ki}\omega_{kl}\check{h}_{lj}
-\sum^{n}_{l=1}\frac{1}{4}\omega_{kj}\omega_{kl}\check{h}_{il}]\circ\pi.
\end{aligned}
\end{equation}
In the second equality, we have used the relation $\pi_{*}(\nabla_{k}\nabla_{k}X_{i})=\check{\nabla}_{k}\check{\nabla}_{k}\check{X_{i}}-\sum^{n}_{l=1}\frac{1}{4}\omega_{ki}\omega_{kl}\check{X}_{l}$.

Then, because $\nabla_{k}X_{k}=\check{\nabla}_{k}\check{X_{k}}-\frac{1}{2}\omega_{kk}U=\check{\nabla}_{k}\check{X_{k}}$, we have
\begin{equation}\label{SecondTerm}
(\nabla_{\nabla_{k}X_{k}}h)(X_{i}, X_{j})=(\check{\nabla}_{\check{\nabla}_{k}X_{k}}\check{h})_{ij}\circ\pi.
\end{equation}

For the third term, we have
\begin{equation}\label{ThirdTerm}
\begin{aligned}
(\nabla_{U}\nabla_{U}h)_{ij}
&=UU(h_{ij})-2U(h(\nabla_{U}X_{i}, X_{j}))-2U(h(X_{i}, \nabla_{U}X_{j}))\\
&\ \ \ \ +h(\nabla_{U}\nabla_{U}X_{i}, X_{j})+2h(\nabla_{U}X_{i}, \nabla_{U}X_{j})+h(X_{i}, \nabla_{U}\nabla_{U}X_{j})\\
&=\sum^{n}_{l=1}\left[\frac{1}{4}\omega_{ik}\omega_{kl}\check{h}_{lj}+\frac{1}{2}\omega_{ik}\omega_{jl}\check{h}_{kl}+ \frac{1}{4}\omega_{jk}\omega_{kl}\check{h}_{il}\right] \circ\pi.
\end{aligned}
\end{equation}
Here, we used the facts that $h_{ij}$, $h(\nabla_{U}X_{i}, X_{j})$, and $h(X_{i}, \nabla_{U}X_{j})$ are constant along fibers, since $h$ is the pull-back of a 2-tensor on the base. We also used $\pi_{*}(\nabla_{U}\nabla_{U}X_{i})=\sum^{n}_{l=1}\frac{1}{4}\omega_{ik}\omega_{kl}\check{X_{l}}$.

Substituting $(\ref{FirstTerm})$, $(\ref{SecondTerm})$, and $(\ref{ThirdTerm})$ into $(\ref{LaplaceTerm})$, we complete the proof of the lemma.
\end{proof}

By using the fundamental equations for the Riemannian curvature tensor in Riemannian submersions (see Theorem 2 in \cite{ONe66} or equation (9.28f) in \cite{Bes87}) and the fact $A_{X}Y=-\frac{1}{2}\omega(\check{X}, \check{Y})U$ for basic vector fields $X$ and $Y$, we have the following relation between the Riemannian curvature tensor on the total space and that on the base.
\begin{lem}\label{CurvatureOnCircleBundles}
\begin{equation}
R_{ijkl}=\check{R}_{ijkl}\circ\pi+ \left(-\frac{1}{2}\omega_{ij}\omega_{kl}+\frac{1}{4}\omega_{jk}\omega_{il}-\frac{1}{4}\omega_{ik}\omega_{jl}\right)\circ\pi,
\end{equation}
and therefore,
\begin{equation}
(\mathring{R}h)_{ij}=(\mathring{\check{R}}\check{h})_{ij}\circ\pi+ \left(-\frac{1}{2}\sum^{n}_{k,l=1}\omega_{ik}\omega_{jl}\check{h}_{kl}
+\frac{1}{4}\sum^{n}_{k,l=1}\omega_{kj}\omega_{il}\check{h}_{kl}\right)\circ\pi,
\end{equation}
where $i$, $j$, $k$, and $l$ run through $1$ to $n$.
\end{lem}

\begin{prop}\label{StabilityOperatorRelation}
\begin{equation}
\langle \nabla^{*}\nabla h-2\mathring{R}h, h\rangle
=\langle \check{\nabla}^{*}\check{\nabla}\check{h}-2\mathring{\check{R}}\check{h}, \check{h}\rangle\circ\pi +\sum^{n}_{i, j, k, l=1}(\omega_{ki}\omega_{kl}\check{h}_{lj}\check{h}_{ij}+\omega_{ik}\omega_{jl}\check{h}_{kl}\check{h}_{ij})\circ\pi,
\end{equation}
\end{prop}

\begin{proof}
By using Lemma \ref{LaplacianOnCircleBundles} and Lemma \ref{CurvatureOnCircleBundles}, we have
\begin{align*}
\langle\nabla^{*}\nabla h-2\mathring{R}h, h\rangle
&=\sum^{n}_{i,j=1}(\nabla^{*}\nabla h-2\mathring{R}h)_{ij}h_{ij}\\
&=\left(\sum^{n}_{i,j=1}(\check{\nabla}^{*}\check{\nabla}\check{h}-2\mathring{\check{R}}\check{h})_{ij}\check{h}_{ij}\right)\circ\pi
   +\sum^{n}_{i,j,k,l=1}[\frac{1}{2}\omega_{ki}\omega_{kl}
  \check{h}_{lj}\check{h}_{ij}+\frac{1}{2}\omega_{kj}\omega_{kl}\check{h}_{li}\check{h}_{ij}\\
&\ \ \ -\frac{1}{2}\omega_{ik}\omega_{jl}\check{h}_{kl}\check{h}_{ij}
  +\omega_{ik}\omega_{jl}\check{h}_{kl}\check{h}_{ij}-\frac{1}{2}\omega_{kj}\omega_{il}\check{h}_{kl}\check{h}_{ij}]\circ\pi\\
&=\langle \check{\nabla}^{*}\check{\nabla}\check{h}-2\mathring{\check{R}}\check{h}, \check{h}\rangle\circ\pi +\sum^{n}_{i,j,k,l=1}(\omega_{ki}\omega_{kl}\check{h}_{lj}\check{h}_{ij}+\omega_{ik}\omega_{jl}\check{h}_{kl}\check{h}_{ij})\circ\pi.
\end{align*}
In the last step above, we have exploited the anti-symmetry (resp. symmetry) of the indices $i, j$ in $\omega_{ij}$
 (resp. $\check{h}_{ij}$).
\end{proof}

\subsection{A Weitzenb\"{o}ck formula and proof of Theorem \ref{CircleBundleInstability}}

In the situation of an Einstein principal circle bundle that we are considering, if the base metric
is not Einstein, a natural symmetric $2$-tensor that presents itself is
\begin{equation}
\check{h}(\check{X}, \check{Y}) :=\sum^{n}_{i=1}\omega(\check{X}, \check{X}_{i})\,\omega(\check{Y}, \check{X}_{i})
      =2 \Ric_{\check{g}}(\check{X}, \check{Y})-2E\check{g}(\check{X}, \check{Y}).
\end{equation}
Then $\check{h}-\frac{\|\omega\|^{2}}{n}\check{g}$ is a nonzero TT-tensor on $B$, since
$\check{s}=nE-\frac{1}{2}\|\omega\|^{2}$ is constant, and so $\delta_{\check{g}}\check{h}=\delta_{\check{g}}(2 \Ric_{\check{g}})=-d\check{s}=0$.
Its pull-back $\pi^{*}\check{h}-\frac{\|\omega\|^{2}}{n}\pi^{*}\check{g}$ is a TT-tensor on $M$.

To prove Theorem \ref{CircleBundleInstability}, we will view the symmetric $2$-tensor $\check{h}$ as a $1$-form on $B$
with values in the cotangent bundle $T^{*}B$, i.e., $\check{h}\in C^{\infty}(B, \wedge^{1}T^{*}B\otimes T^{*}B)$. Using
the Levi-Civita connection of $\check{g}$ on $T^*B$, we obtain an exterior differential $d^{\check{{\nabla}}}$ for
$T^{*}B$-valued differential forms. Let $\delta^{\check{\nabla}}$ denote its formal adjoint operator. In particular,
one has
\begin{equation*}
d^{\check{\nabla}}: C^{\infty}(B, \wedge^{1}T^{*}B\otimes T^{*}B) \rightarrow C^{\infty}(B, \wedge^{2}T^{*}B\otimes T^{*}B),
\end{equation*}
and
\begin{equation*}
\delta^{\check{\nabla}}: C^{\infty}(B, \wedge^{2}T^{*}B\otimes T^{*}B) \rightarrow C^{\infty}(B, \wedge^{1}T^{*}B\otimes T^{*}B).
\end{equation*}
Let $\{\check{X}_{1}, \cdots, \check{X}_{n}\}$ be a local orthonormal basis around a point $b\in B$ such that $(\check{\nabla}_{\check{X}_{i}}\check{X}_{j})(b)=0$ for all $1\leq i, j\leq n$. Then at the point $b$, we have
\begin{equation}
\begin{aligned}
(d^{\check{\nabla}}\check{h})(\check{X}_{i}, \check{X}_{j})(\check{X}_{k})
&=[(\check{\nabla}_{\check{X}_{i}}\check{h})(\check{X}_{j})-(\check{\nabla}_{\check{X}_{j}}\check{h})(\check{X}_{i})](\check{X}_{k})\\
&=[\check{\nabla}_{\check{X}_{i}}(\check{h}(\check{X}_{j}))-\check{\nabla}_{\check{X}_{j}}(\check{h}(\check{X}_{i}))](\check{X}_{k})\\
&=\check{\nabla}_{\check{X}_{i}}(\check{h}(\check{X}_{j}, \check{X}_{k}))-\check{\nabla}_{\check{X}_{j}}(\check{h}(\check{X}_{i}, \check{X}_{k}))\\
&=(\check{\nabla}_{\check{X}_{i}}\check{h})(\check{X}_{j}, \check{X}_{k})-(\check{\nabla}_{\check{X}_{j}}\check{h})(\check{X}_{i}, \check{X}_{k}),
\end{aligned}
\end{equation}
and
\begin{equation}
\begin{aligned}
(\delta^{\check{\nabla}}\check{h})(\check{X}_{j})
&=-\sum^{n}_{i=1}(\check{\nabla}_{\check{X}_{i}}\check{h})(\check{X}_{i})(\check{X}_{j})\\
&=-\sum^{n}_{i=1}\check{\nabla}_{\check{X}_{i}}(\check{h}(\check{X}_{i}))(\check{X}_{j})\\
&=-\sum^{n}_{i=1}\check{\nabla}_{\check{X}_{i}}(\check{h}(\check{X}_{i})(\check{X}_{j}))\\
&=-\sum^{n}_{i=1}\check{\nabla}_{\check{X}_{i}}(\check{h}(\check{X}_{i}, \check{X}_{j}))\\
&=(\delta_{\check{g}}\check{h})(\check{X}_{j}).
\end{aligned}
\end{equation}
The above equalities relate the differential and codifferential on $\check{h}$ regarded as a $T^*B$-valued form
with corresponding operators acting on the symmetric $2$-tensor $\check{h}$. In particular, the condition
$d^{\check{\nabla}} \check{h} = 0$ is equivalent to $\check{h}$ being a Codazzi tensor, and $\check{h}$ is
also divergence free as a $T^*B$-valued form. It is therefore natural to apply the Weitzenb\"{o}ck formula

\begin{equation} \label{Weitzenbock}
(\delta^{\check{\nabla}}d^{\check{\nabla}}+d^{\check{\nabla}}\delta^{\check{\nabla}})\,\check{h}
=\check{\nabla}^{*}\check{\nabla}\check{h}-\mathring{R}^{\check{g}}\check{h}+ \Ric_{\check{g}}\circ \check{h},
\end{equation}
where $(\mathring{\check{R}} \check{h})_{ij}:=\sum\limits^{n}_{k, l=1}\check{R}_{ikjl}\check{h}_{kl}$ and $(\check{\Ric}\circ\check{h})_{ij}:=\sum\limits^{n}_{k=1}\check{\Ric}_{ik}\check{h}_{kj}$.

{\noindent \em Proof of Theorem $\ref{CircleBundleInstability}$}: By the Einstein conditions, one immediately obtains
\begin{equation}  \label{baseEinstein}
\check{\Ric}=\frac{1}{2}\,\check{h}+\frac{\|\omega\|^2}{4}\,\check{g}
\end{equation}
and
\begin{equation}
\check{s}=\frac{1}{2}\, \|\omega\|^{2}+\frac{n}{4}\,\|\omega\|^{2}.
\end{equation}

It then follows from Proposition $\ref{StabilityOperatorRelation}$, the Weitzenb\"{o}ck formula
($\ref{Weitzenbock}$), the assumption $\delta^{\check{\nabla}}\check{h}=\delta_{\check{g}}\check{h}=0$,
and the fact $\tr_{\check{g}}\check{h}=\|\omega\|^{2}$ that we have

\begin{eqnarray*}
\lefteqn{ \langle (\nabla^{*}\nabla-2\mathring{R})(\pi^{*}(\check{h}-\frac{\|\omega\|^{2}}{n}\check{g})), \pi^{*}(\check{h}-\frac{\|\omega\|^{2}}{n}\check{g})\rangle } \\
&=& \{\langle (\check{\nabla}^{*}\check{\nabla}-2\mathring{\check{R}})(\check{h}-\frac{\|\omega\|^{2}}{n}\check{g}), (\check{h}-\frac{\|\omega\|^{2}}{n}\check{g})\rangle
 +\omega_{ki}\,\omega_{kl}(\check{h}-\frac{\|\omega\|^{2}}{n}\check{g})_{lj}(\check{h}-\frac{\|\omega\|^{2}}{n}\check{g})_{ij} \\
&   & \,\,\, + \omega_{ik}\,\omega_{jl}(\check{h}-\frac{\|\omega\|^{2}}{n}\check{g})_{kl}(\check{h}-\frac{\|\omega\|^{2}}{n}\check{g})_{ij}
\ \} \circ \pi  \\
&=&  \{\langle \check{\nabla}^{*}\check{\nabla}\check{h}, \check{h}\rangle - 2\langle \mathring{\check{R}}\check{h}, \check{h} \rangle + 4\, \frac{\|\omega\|^{2}}{n}\langle \check{\Ric}, \check{h} \rangle - 2\,\frac{\|\omega\|^{4}}{n^{2}}\check{s}+2\, \tr_{\check{g}}(\check{h}\circ \check{h}\circ \check{h})\\
&    &\ \ \ \ - \frac{4\|\omega\|^{2} \|\check{h}\|^{2}}{n}+\frac{2\|\omega\|^{6}}{n^{2}}\}\circ \pi\\
&=& \{2\langle \delta^{\check{\nabla}}d^{\check{\nabla}}\check{h}, \check{h} \rangle - \langle \check{\nabla}^{*}\check{\nabla} \check{h}, \check{h} \rangle - 2\langle \check{\Ric}\circ\check{h}, \check{h} \rangle + 4\,\frac{\|\omega\|^{2}}{n}\langle \frac{1}{2}\check{h}+\frac{\|\omega\|^{2}}{4}\check{g}, \check{h} \rangle \\
&   &\ \ \ \ -\frac{2\|\omega\|^{4}}{n^{2}}(\frac{1}{2}\|\omega\|^{2}+\frac{n}{4}\|\omega\|^{2}) + 2\tr_{\check{g}}(\check{h}\circ \check{h}\circ \check{h}) - \frac{4\|\omega\|^{2} \|\check{h}\|^{2}}{n}+\frac{2\|\omega\|^{6}}{n^{2}}\} \circ \pi\\
&=  & \{2\langle \delta^{\check{\nabla}}d^{\check{\nabla}}\check{h}, \check{h} \rangle - \langle \check{\nabla}^{*}\check{\nabla} \check{h}, \check{h} \rangle - \tr_{\check{g}}(\check{h}\circ \check{h}\circ \check{h}) - \frac{\|\omega\|^{2}\|\check{h}\|^{2}}{2} + \frac{2\|\omega\|^{2}\|\check{h}\|^{2}}{n} + \frac{\|\omega\|^{6}}{n}\\
&    &\ \ \ \ - \frac{\|\omega\|^{6}}{n^{2}} - \frac{\|\omega\|^{6}}{2n} + 2\tr_{\check{g}}(\check{h}\circ \check{h}\circ \check{h}) - \frac{4\|\omega\|^{2} \|\check{h}\|^{2}}{n}+\frac{2\|\omega\|^{6}}{n^{2}}\}\circ \pi\\
&=&  \{2\langle \delta^{\check{\nabla}}d^{\check{\nabla}}\check{h}, \check{h} \rangle - \langle \check{\nabla}^{*}\check{\nabla} \check{h}, \check{h} \rangle + \tr_{\check{g}}(\check{h}\circ \check{h}\circ \check{h}) - \frac{\|\omega\|^{2}\|\check{h}\|^{2}}{2} - \frac{2\|\omega\|^{2}\|\check{h}\|^{2}}{n}\\
&      &\ \ \  \ + \frac{\|\omega\|^{6}}{2n} + \frac{\|\omega\|^{6}}{n^{2}}\}\circ\pi.
\end{eqnarray*}
\noindent{Note that we used equation (\ref{baseEinstein}) to get the fourth equality in the above calculation.} \hspace{1.5cm} $\Box$

\subsection{Proof of Corollary \ref{HarmonicCurvatureBaseInstability}}
Now we assume in addition that $\check{g}$ has harmonic curvature but is not Einstein.
Then $d^{\check{\nabla}}\check{h}=2d^{\check{\nabla}}\check{\Ric}=0$ and so by Theorem $\ref{CircleBundleInstability}$, we have
\begin{eqnarray*}
\lefteqn{\int_{M}\langle((\nabla)^{*}\nabla-2\mathring{R}^{g})\pi^{*}\left(\check{h}-\frac{\|\omega\|^{2}}{n}\check{g}\right), \pi^{*}\left(\check{h}-\frac{\|\omega\|^{2}}{n}\check{g}\right)\rangle \,d\vol_{g} }\\
&= & \int_{M} \{- \langle \check{\nabla}^{*}\check{\nabla} \check{h}, \check{h} \rangle + tr_{\check{g}}(\check{h}\circ \check{h}\circ \check{h}) - \frac{\|\omega\|^{2}\|\check{h}\|^{2}}{2} - \frac{2\|\omega\|^{2}\|\check{h}\|^{2}}{n} + \frac{\|\omega\|^{6}}{2n} + \frac{\|\omega\|^{6}}{n^{2}}\}\circ\pi \, d\vol_{g}\\
&= & 2\pi \int_{B}\{ -\langle \nabla\check{h}, \nabla\check{h} \rangle + tr_{\check{g}}(\check{h}\circ \check{h}\circ \check{h}) - \frac{\|\omega\|^{2}\|\check{h}\|^{2}}{2} - \frac{2\|\omega\|^{2}\|\check{h}\|^{2}}{n} + \frac{\|\omega\|^{6}}{2n} + \frac{\|\omega\|^{6}}{n^{2}} \} \,d\vol_{\check{g}}\\
&  \leq & 2\pi \int_{B}\{ tr_{\check{g}}(\check{h}\circ \check{h}\circ \check{h}) - \frac{\|\omega\|^{2}\|\check{h}\|^{2}}{2} - \frac{2\|\omega\|^{2}\|\check{h}\|^{2}}{n} + \frac{\|\omega\|^{6}}{2n} + \frac{\|\omega\|^{6}}{n^{2}}\} \, d\vol_{\check{g}}.\\
\end{eqnarray*}
Thus, in order to prove Corollary \ref{HarmonicCurvatureBaseInstability}, it suffices to show that
\begin{equation}
\tr_{\check{g}}(\check{h}\circ \check{h}\circ \check{h}) - \frac{\|\omega\|^{2}\|\check{h}\|^{2}}{2} - \frac{2\|\omega\|^{2}\|\check{h}\|^{2}}{n} + \frac{\|\omega\|^{6}}{2n} + \frac{\|\omega\|^{6}}{n^{2}}\leq 0,
\end{equation}
on $B$ and the inequality is strict on an open subset of $B$.

Now let us work at a point $b\in B$. We can choose an orthonormal basis of $T_{b}B$ such that at $b$, $\omega$ has its block diagonal standard form as a  skew-symmetric bilinear form on $T_{b}B$. Then this basis diagonalizes the symmetric bilinear form $\check{h}$ on $T_{b}B$, i.e., if $n=2m$, then $\check{h}={\rm diag}\{a^{2}_{1}, a^{2}_{1}, \cdots, a^{2}_{m}, a^{2}_{m}\}$, and if $n=2m+1$, then $\check{h}={\rm diag}\{a^{2}_{1}, a^{2}_{1}, \cdots, a^{2}_{m}, a^{2}_{m}, 0\}$.

Let $b_{i}=a^{2}_{i}$ for $1\leq i\leq m$. Then
$\|\omega\|^{2} =  2\sum\limits^{m}_{i=1}b_{i}=4E,$
and $\|\check{h}\|^{2} = 2\sum\limits^{m}_{i=1}b^{2}_{i}$.
Set $t_{i}=\frac{b_{i}}{2E}$ for $1\leq i\leq m$. Then $\sum\limits^{m}_{i=1}t_{i}=1$, and at $b$
\begin{eqnarray*}
\lefteqn{ \tr_{\check{g}}(\check{h}\circ \check{h}\circ \check{h}) - \frac{\|\omega\|^{2}\|\check{h}\|^{2}}{2} - \frac{2\|\omega\|^{2}\|\check{h}\|^{2}}{n} + \frac{\|\omega\|^{6}}{2n} + \frac{\|\omega\|^{6}}{n^{2}} } \\
&=&  2\left(\sum^{m}_{i=1}b^{3}_{i}-(\sum^{m}_{i=1}b_{i})(\sum^{m}_{j=1}b^{2}_{j})\right)
+\left(\frac{4}{n}+\frac{8}{n^{2}}\right)\left(\sum^{m}_{i=1}b_{i}\right)^{3}-\frac{8}{n}(\sum^{m}_{i=1}b_{i})(\sum^{m}_{j=1}b^{2}_{j})\\
&= &   2(2E)^{3}f(t_{1}, \cdots, t_{m}),
\end{eqnarray*}
where $f(t_{1}, \cdots, t_{m}):=\sum\limits^{m}_{i=1}t^{3}_{i}-\left(1+\frac{4}{n}\right)\sum\limits^{m}_{i=1}t^{2}_{i}+\frac{2}{n}+\frac{4}{n^{2}}$.

{\em Case 1: $n=2m$.}
\begin{align*}
f(t_{1}, \cdots, t_{m})
&=\sum^{m}_{i=1}\left(t_{i}-\frac{1}{m}+\frac{1}{m}\right)^{3}-\left(1+\frac{2}{m}\right)\sum^{m}_{i=1}\left(t_{i}-\frac{1}{m}+\frac{1}{m}\right)^{2}
+\frac{1}{m}+\frac{1}{m^{2}}\\
&=\sum^{m}_{i=1}\left(t_{i}-\frac{1}{m}\right)^{2}(t_{i}-1)\leq 0,
\end{align*}
since $0\leq t_{i}\leq 1$ and $\sum\limits^{m}_{i=1}t_{i}=1$.

Equality holds iff $\left(t_{i}-\frac{1}{m}\right)^{2}(t_{1}-1)=0$ for all $1\leq i\leq m$. If $t_{i}=1$ for some $i$, all other $t_{j}=0$ and we get a contradiction unless $m=1$. So if $m>1$, equality holds iff $t_{i}=\frac{1}{m}$ for all $i$. On the other hand, if $m=1$ then
the base manifold, which has constant scalar curvature, is $2$-dimensional and so must be Einstein. So $m>1$ holds, and
$$\tr_{\check{g}}(\check{h}\circ \check{h}\circ \check{h}) - \frac{\|\omega\|^{2}\|\check{h}\|^{2}}{2} - \frac{2\|\omega\|^{2}\|\check{h}\|^{2}}{n} + \frac{\|\omega\|^{6}}{2n} + \frac{\|\omega\|^{6}}{n^{2}}\leq 0.$$
Furthermore, for at least one point $b\in B$, the inequality must be strict, for otherwise, $(B, \check{g})$ would be Einstein, which contradicts our assumption. But then strict inequality must also hold in a neighborhood of $b$ and we are done.

\medskip

{\em Case 2: n=2m+1.} We expand $f$ about $t=\frac{2}{2m+1}$ instead and obtain
\begin{equation*}
f(t_{1}, \cdots, t_{m})=\sum^{m}_{i=1}\left(t_{i}-\frac{2}{2m+1}\right)^{2}(t_{i}-1)-\frac{2+4m}{(2m+1)^{3}}
<\sum^{m}_{i=1}\left(t_{i}-\frac{2}{2m+1}\right)^{2}(t_{i}-1)\leq 0,
\end{equation*}
since $t_{i}\leq1$. Thus
$$\tr_{\check{g}}(\check{h}\circ \check{h}\circ \check{h}) - \frac{\|\omega\|^{2}\|\check{h}\|^{2}}{2} - \frac{2\|\omega\|^{2}\|\check{h}\|^{2}}{n} + \frac{\|\omega\|^{6}}{2n} + \frac{\|\omega\|^{6}}{n^{2}}< 0,$$
on $B$. This completes the proof of Corollary $\ref{HarmonicCurvatureBaseInstability}$.  \hspace{7.5cm} $\Box$

\subsection{Proof of Corollary \ref{StableCircleBundleRigidity}}

By Corollary $\ref{HarmonicCurvatureBaseInstability}$, $\check{g}$ has to be Einstein. Then Theorem 9.76 in \cite{Bes87} shows that $(B, \check{g})$ has a compatible almost complex structure $J^{\prime} := \frac{\sqrt{n}}{\|\omega\|}J$, where the automorphism $J: TB\rightarrow TB$ is defined by $\omega(\check{X}, \check{Y})=\check{g}(J\check{X}, \check{Y})$ for any pair of vector fields $\check{X}, \check{Y}\in C^{\infty}(B, TB)$.
 Let $\omega^{\prime}$ be the corresponding K\"{a}hler form. Then by definition $\omega=\frac{\|\omega\|}{\sqrt{n}}\omega^{\prime}$. By the resolution of the Goldberg conjecture \cite{Sek87}, the almost complex structure $J^{\prime}$ is actually integrable, i.e., $(B, \check{g}, J^{\prime})$ is K\"{a}hler Einstein.

If the base manifold $B$ has $b_{2}>1$, then there exists a real harmonic  $2$-form $\eta$ of type $(1, 1)$
orthogonal to $\omega^{\prime}$ with respect to $\check{g}$. By composing with the complex structure $J^{\prime}$
we obtain a symmetric $2$-tensor $\check{h}$ as $\check{h}(\check{X}, \check{Y})=\eta(J^{\prime}\check{X}, \check{Y})$.
Because $J^{\prime}$ is parallel, $\check{h}$ is a TT-tensor on $(B, \check{g})$. By further straightforward calculations,
$h=\pi^{*}\check{h}$ is a TT-tensor on $(M, g)$.

Moreover, again because $J^{\prime}$ is parallel, by
\begin{equation*}
0=(\Delta_{d}\,\eta)_{ij}=(\check{\nabla}^{*}\check{\nabla}\eta)_{ij}-2\check{R}_{ikjl}\eta_{kl}+\check{\Ric}_{ik}\eta_{kj}+\check{\Ric}_{jk}\eta_{ki},
\end{equation*}
we have
\begin{equation*}
\check{\nabla}^{*}\check{\nabla}\check{h}-2\mathring{\check{R}}\check{h}+2\check{\Ric}\circ\check{h}=0.
\end{equation*}
On the other hand, by the Einstein conditions, $\check{\Ric}=\frac{n+2}{4n}\,\|\omega\|^{2}\check{g}$. Thus,
\begin{equation}\label{EinsteinBaseInstabilityDirection}
\check{\nabla}^{*}\check{\nabla}\check{h}-2\mathring{\check{R}}\check{h}=-\frac{n+2}{2n}\|\omega\|^{2}\check{h}.
\end{equation}

Finally, by Proposition $\ref{StabilityOperatorRelation}$, the equation $(\ref{EinsteinBaseInstabilityDirection})$, and the fact $\omega=\frac{\|\omega\|}{\sqrt{n}}\omega^{\prime}$, we have
\begin{eqnarray*}
\langle\nabla^{*}\nabla h-2\mathring{R}h, h\rangle
&=& \langle\check{\nabla}^{*}\check{\nabla}\check{h}-2\mathring{\check{R}}\check{h}, \check{h}\rangle\circ\pi+
\frac{\|\omega\|^{2}}{n}(\omega^{\prime}_{ki}\omega^{\prime}_{kl}\check{h}_{lj}\check{h}_{ij}
+\omega^{\prime}_{ik}\omega^{\prime}_{jl}\check{h}_{kl}\check{h}_{ij})\circ\pi\\
&=& \langle -\frac{n+2}{2n}\|\omega\|^{2}\check{h}, \check{h}\rangle\circ\pi+\frac{\|\omega\|^{2}}{n}(\|\check{h}\|^{2}+\langle \check{h}(J^{\prime}\cdot, J^{\prime}\cdot), \check{h}\rangle)\circ\pi\\
&\leq& -\frac{n+2}{2n}\|\omega\|^{2}\|\check{h}\|^{2}\circ\pi+\frac{2}{n}\|\omega\|^{2}\|\check{h}\|^{2}\circ\pi\\
&=&-\left(\frac{1}{2}-\frac{1}{n}\right)\|\omega\|^{2}\|\check{h}\|^{2}\circ\pi,
\end{eqnarray*}
which is negative if $n\geq 4$, and this contradicts the assumption of stability. Note that in the second last step above, we have used $\langle \check{h}(J^{\prime}\cdot, J^{\prime}\cdot), \check{h}\rangle \leq \|\check{h}\|^{2}$.

When $n=2$, the base manifold $B = S^2$ so certainly $b_{2}=1$. (The total space $(M, g)$ is then a $3$-dimensional
Einstein manifold and so has constant sectional curvature. Therefore, $g$ is stable.)
 This completes the proof of Corollary \ref{StableCircleBundleRigidity}. \hspace{9cm} $\Box$


\section{Acknowlegments}

The first author would like to express his gratitude to Professors Xianzhe Dai and Guofang Wei for many inspiring discussions and their constant encouragement and support. He would also like to thank Chenxu He for his interests and useful discussions. During the Fall term of
2017-2018, he has been supported by a Fields Postdoctoral Fellowship. He thanks the Fields Institute for Research in Mathematical Sciences for providing an excellent research environment.

The second author acknowledges partial support by NSERC  Grant No. OPG0009421.

Both authors thank Professor Peng Lu for his comments on an earlier version of the paper.



\end{document}